\documentclass[11pt, notitlepage]{article}   

\usepackage{amsmath,amsthm,amsfonts}   

\usepackage{amscd}
\usepackage{amsfonts}
\usepackage{amssymb}
\usepackage{color}      
\usepackage{epsfig}
\usepackage{graphicx}           

\usepackage{tikz}

\theoremstyle{plain}
\newtheorem{theorem}{Theorem}[section]
\newtheorem{lemma}[theorem]{Lemma}
\newtheorem{corollary}[theorem]{Corollary}
\newtheorem{proposition}[theorem]{Proposition}
\newtheorem{observation}[theorem]{Observation}
\newtheorem{remark}[theorem]{Remark}

\newtheorem{question}[theorem]{Question}

\theoremstyle{definition}

\newcommand{\diam}{\textnormal{diam}}

\def\finf{\mathop{{\rm I}\kern -.27 em {\rm F}}\nolimits}

\begin{document}


\title{The fractional $k$-truncated metric dimension of graphs}

\author{{\bf{Eunjeong Yi}}\\
\small Texas A\&M University at Galveston, Galveston, TX 77553, USA\\
{\small\em yie@tamug.edu}}

\maketitle

\date{}

\begin{abstract}
The metric dimension, $\dim(G)$, and the fractional metric dimension, $\dim_f(G)$, of a graph $G$ have been studied extensively. 
Let $G$ be a graph with vertex set $V(G)$, and let $d(x,y)$ denote the length of a shortest $x-y$ path in $G$. Let $k$ be a positive integer. For any $x,y \in V(G)$, let $d_k(x,y)=\min\{d(x,y), k+1\}$ and let $R_k\{x,y\}=\{z\in V(G): d_k(x,z) \neq d_k(y,z)\}$. A set $S \subseteq V(G)$ is a \emph{$k$-truncated resolving set} of $G$ if $|S \cap R_k\{x,y\}| \ge 1$ for any distinct $x,y\in V(G)$, and the \emph{$k$-truncated metric dimension} $\dim_k(G)$ of $G$ is the minimum cardinality over all $k$-truncated resolving sets of $G$. For a function $g$ defined on $V(G)$ and for $U \subseteq V(G)$, let $g(U)=\sum_{s\in U}g(s)$. A real-valued function $g:V(G) \rightarrow[0,1]$ is a \emph{$k$-truncated resolving function} of $G$ if $g(R_k\{x,y\}) \ge 1$ for any distinct $x, y\in V(G)$, and the \emph{fractional $k$-truncated metric dimension} $\dim_{k,f}(G)$ of $G$ is $\min\{g(V(G)): g \mbox{ is a $k$-truncated resolving function of }G\}$. Note that $\dim_{k,f}(G)$ reduces to $\dim_k(G)$ if the codomain of $k$-truncated resolving functions is restricted to $\{0,1\}$, and $\dim_{k,f}(G)=\dim_f(G)$ if $k$ is at least the diameter of $G$. In this paper, we study the fractional $k$-truncated metric dimension of graphs. For any connected graph $G$ of order $n\ge2$, we show that $1 \le \dim_{k,f}(G) \le \frac{n}{2}$; we characterize $G$ satisfying $\dim_{k,f}(G)$ equals $1$ and $\frac{n}{2}$, respectively. We examine $\dim_{k,f}(G)$ of some graph classes. We also show the existence of non-isomorphic graphs $G$ and $H$ such that $\dim_k(G)=\dim_k(H)$ and $\dim_{k,f}(G)\neq \dim_{k,f}(H)$, and we examine the relation among $\dim(G)$, $\dim_f(G)$, $\dim_k(G)$ and $\dim_{k,f}(G)$. We conclude the paper with some open problems.
\end{abstract}



\section{Introduction}

Let $G$ be a finite, simple, undirected, and connected graph with vertex set $V(G)$ and edge set $E(G)$. The \emph{distance} between two vertices $x, y \in V(G)$, denoted by $d(x, y)$, is the minimum number of edges on a path connecting $x$ and $y$ in $G$. The \emph{diameter}, $\diam(G)$, of $G$ is $\max\{d(x,y): x,y \in V(G)\}$. Let $\mathbb{Z}^+$ denote the set of positive integers. For $k\in\mathbb{Z}^+$ and for two vertices $x,y\in V(G)$, let $d_k(x,y)=\min\{d(x,y),k+1\}$. 

Metric dimension, introduced in~\cite{harary} and~\cite{slater}, is a graph parameter that has been studied extensively. For distinct $x,y\in V(G)$, let $R\{x,y\}=\{z\in V(G): d(x,z) \neq d(y,z)\}$. A vertex subset $S\subseteq V(G)$ is a \emph{resolving set} of $G$ if $|S \cap R\{x,y\}| \ge 1$ for any pair of distinct $x,y \in V(G)$, and the \emph{metric dimension} $\dim(G)$ of $G$ is the minimum cardinality over all resolving sets of $G$. For $k\in\mathbb{Z}^+$ and for distinct $x,y \in V(G)$, let $R_k\{x,y\}=\{z\in V(G): d_k(x,z) \neq d_k(y,z)\}$. A vertex subset $S\subseteq V(G)$ is a \emph{$k$-truncated resolving set} (also called a distance-$k$ resolving set) of $G$ if $|S \cap R_k\{x,y\}| \ge 1$ for any pair of distinct $x,y \in V(G)$, and the \emph{$k$-truncated metric dimension} (also called the distance-$k$ dimension) $\dim_k(G)$ of $G$ is the minimum cardinality over all $k$-truncated resolving sets of $G$. Notice that $\dim_k(G)=\dim(G)$ if $k\ge \diam(G)-1$. The metric dimension of a metric space $(V, d_k)$ is studied in~\cite{beardon}. The $k$-truncated metric dimension corresponds to the $(1, k+1)$-metric dimension in~\cite{moreno1, moreno2}. We note that $\dim_1(G)$ is also called the adjacency dimension of $G$ in~\cite{adim}. For detailed results on $\dim_k(G)$, we refer to~\cite{distKdim}. It is known that determining the metric dimension and the $k$-truncated metric dimension of a general graph are NP-hard problems; see~\cite{moreno2, Juan, NP, tree2}.

The fractionalization of various graph parameters has been extensively studied (see~\cite{frac}). For definition and a formulation of fractional metric dimension as the optimal solution to a linear programming problem by relaxing a condition of the integer programming problem for metric dimension, see~\cite{oellermann_1, oellermann_2}. The fractional metric dimension of graphs was officially studied in~\cite{fracdim1}. For a function $g$ defined on $V(G)$ and for $U \subseteq V(G)$, let $g(U)=\sum_{s \in U} g(s)$. A real-valued function $g: V(G) \rightarrow [0,1]$ is a \emph{resolving function} of $G$ if $g(R\{x,y\}) \ge 1$ for any distinct vertices $x,y\in V(G)$. The \emph{fractional metric dimension}, $\dim_f(G)$, of $G$ is $\min\{g(V(G)): g \mbox{ is a resolving function of }G\}$. Notice that $\dim_f(G)$ reduces to $\dim(G)$ if the codomain of resolving functions is restricted to $\{0,1\}$. 

For $k\in\mathbb{Z}^+$, a real-valued function $h:V(G)\rightarrow [0,1]$ is a \emph{$k$-truncated resolving function} (also called a distance-$k$ resolving function) of $G$ if $h(R_k\{x,y\}) \ge 1$ for any pair of distinct $x,y\in V(G)$. The \emph{fractional $k$-truncated metric dimension} (also called the fractional distance-$k$ dimension) of $G$, denoted by $\dim_{k,f}(G)$, is $\min\{h(V(G)): h \mbox{ is a $k$-truncated resolving function of }G\}$. Note that $\dim_{k,f}(G)=\dim_k(G)$ if the codomain of $k$-truncated resolving functions is restricted to $\{0,1\}$, $\dim_{k,f}(G)=\dim_f(G)$ if $k \ge \diam(G)-1$, and $\dim_{k,f}(G)=\dim(G)$ if $k \ge \diam(G)-1$ and the codomain of $k$-truncated resolving functions is restricted to $\{0,1\}$. 

In this paper, we initiate the study of the fractional $k$-truncated metric dimension of graphs. Let $P_n$, $C_n$ and $K_n$, respectively, denote the path, the cycle and the complete graph on $n$ vertices. In Section~\ref{Sec_bounds}, for any connected graph $G$ of order $n\ge 2$ and for any $k\in\mathbb{Z}^+$, we show that $1 \le \dim_{k,f}(G) \le \frac{n}{2}$ and we characterize $G$ satisfying $\dim_{k,f}(G)$ equals $1$ and $\frac{n}{2}$, respectively. Along the way, we state some useful observations. In Section~\ref{Sec_graphs}, we examine $\dim_{k,f}(G)$ when $G$ is the Petersen graph, a wheel graph, a fan graph, a complete multipartite graph, a cycle, a grid graph or a tree. In Section~\ref{Sec_comparison}, we examine the relation among $\dim(G)$, $\dim_f(G)$, $\dim_k(G)$ and $\dim_{k,f}(G)$. We show the existence of non-isomorphic graphs $G$ and $H$ such that $\dim_k(G)=\dim_k(H)$ and $\dim_{k,f}(G)\neq \dim_{k,f}(H)$. Based on the construction in~\cite{broadcast}, we also show the existence of two connected graphs $H$ and $G$ with $H \subset G$ such that $\frac{\dim_{k,f}(H)}{\dim_{k,f}(G)}$ can be arbitrarily large. We conclude the paper with some open problems.


\section{Some observations and bounds on $\dim_{k,f}(G)$}\label{Sec_bounds}

In this section, for any connected graph $G$ of order $n\ge2$ and for any $k\in\mathbb{Z}^+$, we show that $1\le \dim_{k,f}(G)\le \frac{n}{2}$; we characterize $G$ satisfying $\dim_{k,f}(G)$ equals $1$ and $\frac{n}{2}$, respectively.

We begin with some observations. For $v\in V(G)$, let $N(v)=\{w\in V(G): vw\in E(G)\}$ and $N[v]=N(v) \cup \{v\}$. More generally, for $k\in \mathbb{Z}^+$ and for $v\in V(G)$, let $N_k[v]=\{w\in V(G): d(v, w)\le k\}$ and $N_k(v)=\{w\in V(G): d(v, w)= k\}$. Two vertices $x,y\in V(G)$ are called \emph{twins} if $N(x)-\{y\}=N(y)-\{x\}$. Hernando et al.~\cite{Hernando} observed that the twin relation is an equivalence relation and that an equivalence class under it, called a twin equivalence class, induces either a clique or an independence set. 

\begin{observation}\label{obs_twin}
Let $x,y\in V(G)$ with $x\neq y$, and let $k\in \mathbb{Z}^+$. 
\begin{itemize}
\item[(a)] \emph{\cite{fracdim_yi}} If $x$ and $y$ are twins, then $g(x)+g(y) \ge 1$ for any resolving function $g$ of $G$.
\item[(b)] If $x$ and $y$ are twins, then $h(x)+h(y) \ge 1$ for any $k$-truncated resolving function $h$ of $G$.
\item[(c)] For any $k\ge1$, $R_k\{x,y\}=(N_k[x] \cup N_k[y])-\cup_{i=1}^{k}(N_i(x) \cap N_i(y))$.
\end{itemize}
\end{observation}

\begin{observation}\label{obs_bounds}
Let $G$ be a non-trivial connected graph, and let $k,k'\in \mathbb{Z}^+$. Then
\begin{itemize}
\item[(a)] \emph{\cite{fracdim1}} $\dim_f(G)\le \dim(G)$;
\item[(b)] \emph{\cite{beardon, moreno1}} if $k> k'$, then $\dim(G)\le \dim_k(G) \le \dim_{k'}(G)\le\dim_1(G)$;
\item[(c)] $\dim_f(G)\le \dim_{k,f}(G) \le \dim_k(G)$;
\item[(d)] if $k> k'$, then $\dim_f(G)\le \dim_{k,f}(G) \le \dim_{k',f}(G)\le\dim_{1,f}(G)\le \dim_1(G)$.
\end{itemize}
\end{observation}

\begin{observation}\label{obs_diam}
Let $G$ be a connected graph with $\diam(G)=d$, and let $k\in\mathbb{Z}^+$. 
\begin{itemize}
\item[(a)] \emph{\cite{moreno1}} If $k\ge d-1$, then $\dim_k(G)=\dim(G)$.
\item[(b)] If $k \ge d-1$, then $\dim_{k,f}(G)=\dim_f(G)$.
\end{itemize}
\end{observation}

Next, we recall some results involving the bounds of the $k$-truncated metric dimension and the fractional metric dimension of graphs.

\begin{theorem}\label{characterization_mixed}
Let $G$ be a connected graph of order $n\ge2$, and let $k\in\mathbb{Z}^+$. Then 
\begin{itemize}
\item[(a)] \emph{\cite{moreno1}} $1\le\dim_k(G)\le n-1$, and $\dim_k(G)=1$ if and only if $G\in\cup_{i=2}^{k+2}\{P_i\}$;
\item[(b)] \emph{\cite{fracdim1, frac_kang}} $1 \le \dim_f(G)\le \frac{n}{2}$, and $\dim_f(G)=\frac{n}{2}$ if and only if there exists a bijection $\phi:V(G) \rightarrow V(G)$ such that $\phi(v)\neq v$ and $|R\{v, \phi(v)\}|=2$ for all $v\in V(G)$;
\item[(c)] \emph{\cite{fracsdim}} $\dim_f(G)=1$ if and only if $G=P_n$.
\end{itemize}
\end{theorem}

For the characterization of connected graphs $G$ of order $n$ satisfying $\dim_k(G)=n-2$ and $\dim_k(G)=n-1$ respectively, see~\cite{distKdim, thesis_cs}. For an explicit characterization of graphs $G$ satisfying $\dim_f(G)=\frac{|V(G)|}{2}$, we recall the following construction from~\cite{fracdim2}. Let $\mathcal{K}=\{K_a: a \ge 2\}$ and $\overline{\mathcal{K}}=\{\overline{K}_b: b \ge 2\}$, where $\overline{G}$ denotes the complement of a graph $G$. Let $H[\mathcal{K} \cup \overline{\mathcal{K}}]$ be the family of graphs obtained from a connected graph $H$ by replacing each vertex $u_i \in V(H)$ by a graph $H_i \in \mathcal{K} \cup \overline{\mathcal{K}}$, and each vertex in $H_i$ is adjacent to each vertex in $H_j$ if and only if $u_iu_j \in E(H)$.

\begin{theorem}~\emph{\cite{fracdim2}}\label{n_2}
Let $G$ be a connected graph of order at least two. Then $\dim_f(G)=\frac{|V(G)|}{2}$ if and only if $G \in H[\mathcal{K} \cup \overline{\mathcal{K}}]$ for some connected graph $H$.
\end{theorem}

Next, we obtain the bounds on $\dim_{k,f}(G)$.

\begin{proposition}\label{frac_kdim_bounds}
For any connected graph $G$ of order $n \ge 2$ and for any $k\in\mathbb{Z}^+$, $1\le \dim_{k,f}(G)\le \frac{n}{2}$.
\end{proposition}

\begin{proof}
Let $k\in\mathbb{Z}^+$, and let $G$ be a connected graph of order $n\ge2$. By definition, $\dim_{k,f}(G)\ge1$. If $g:V(G) \rightarrow [0,1]$ is a function defined by $g(v)=\frac{1}{2}$ for each $v\in V(G)$, then $R_k\{x,y\}\supseteq \{x,y\}$ and $g(R_k\{x,y\}) \ge g(x)+g(y)=1$ for any distinct $x,y\in V(G)$; thus, $g$ is a $k$-truncated resolving function of $G$ with $g(V(G))=\frac{n}{2}$. So, $\dim_{k,f}(G) \le \frac{n}{2}$.~\hfill 
\end{proof}

Next, we characterize connected graphs $G$ satisfying $\dim_{k,f}(G)=1$ for any $k\in\mathbb{Z}^+$. 

\begin{theorem}\label{frac_kdim=1}
Let $G$ be a non-trivial connected graph, and let $k\in\mathbb{Z}^+$. Then $\dim_{k,f}(G)=1$ if and only if $G \in \cup_{i=2}^{k+2}\{P_i\}$.
\end{theorem}

\begin{proof}
Let $G$ be a connected graph of order $n\ge2$, and let $k\in\mathbb{Z}^+$. 

($\Leftarrow$) Let $G \in \cup_{i=2}^{k+2}\{P_i\}$. Then $1=\dim_f(G) \le \dim_{k,f}(G) \le \dim_k(G)=1$ by Observation~\ref{obs_bounds}(c) and Theorem~\ref{characterization_mixed}(a)(c). So, $\dim_{k,f}(G)=1$.

($\Rightarrow$) Let $\dim_{k,f}(G)=1$. By Observation~\ref{obs_bounds}(c) and Theorem~\ref{characterization_mixed}(c), $\dim_{k,f}(G)\ge \dim_f(G)\ge 1$ and $\dim_f(G)=1$ if and only if $G=P_n$. So, if $G\neq P_n$, then $\dim_{k,f}(G)>1$. Now, suppose $G=P_n$, and let $P_n$ be a path given by $u_1, u_2, \ldots, u_n$. Let $g:V(P_n)\rightarrow [0,1]$ be any minimum $k$-truncated resolving function of $P_n$. If $n\le k+2$, then $\dim_{k,f}(P_n)=1$ as shown above. So, suppose $n\ge k+3$; we show that $\dim_{k,f}(P_n)>1$. 

First, let $k+3\le n \le 2k+3$. Then $R_k\{u_1, u_2\}=\cup_{i=1}^{k+2}\{u_i\}$, $R_k\{u_{n-1}, u_n\}=\cup_{i=n-(k+1)}^{n}\{u_i\}$, and $R_k\{u_{i-1}, u_{i+1}\}=V(P_n)-\{u_i\}$ for each $i\in\{n-(k+1), \ldots, k+2\}$. So, $g(R_k\{u_1, u_2\}) \ge \sum_{i=1}^{k+1} g(u_i)\ge1$, $g(R_k\{u_{n-1}, u_n\})\ge \sum_{i=n-(k+1)}^{n} g(u_i) \ge1$, and $g(R_k\{u_{i-1}, u_{i+1}\})=g(V(P_n))-g(u_i)\ge1$ for each $i\in\{n-(k+1), \ldots, k+2\}$. By summing over the $(6+2k-n)$ inequalities, we have $(5+2k-n)g(V(P_n)) \ge 6+2k-n$, i.e., $g(V(P_n)) \ge \frac{6+2k-n}{5+2k-n}$; thus, $\dim_{k,f}(P_n) \ge \frac{6+2k-n}{5+2k-n}>1$ (also see Theorem~\ref{frac_distK_path}(b)). 

Second, let $n\ge 2k+4$. Then $R_k\{u_1, u_2\}=\cup_{i=1}^{k+2}\{u_i\}$ and $R_k\{u_{n-1}, u_n\}=\cup_{i=n-k-1}^{n}\{u_i\}$. So, $g(R_k\{u_1, u_2\})=\sum_{i=1}^{k+2}g(u_i)\ge 1$ and  $g(R_k\{u_{n-1}, u_n\})=\sum_{i=n-k-1}^{n}g(u_i)\ge 1$. Since $n-k-1\ge2k+4-k-1= k+3$, $g(V(P_n)) \ge \sum_{i=1}^{k+2}g(u_i)+\sum_{i=n-k-1}^{n}g(u_i)\ge 2$, which implies $\dim_{k,f}(P_n)\ge 2$.~\hfill
\end{proof}

Next, via a proof technique used in~\cite{fracdim2}, we characterize connected graphs $G$ satisfying $\dim_{k,f}(G)=\frac{|V(G)|}{2}$ for any $k\in\mathbb{Z}^+$.

\begin{theorem}\label{frac_adim_n2}
Let $G$ be a connected graph of order $n \ge 2$. Then $\dim_{1,f}(G)=\frac{n}{2}$ if and only if $G \in H[\mathcal{K} \cup \overline{\mathcal{K}}]$ for some connected graph $H$.
\end{theorem}

\begin{proof}
Let $G$ be a connected graph of order $n \ge 2$. 

($\Leftarrow$) Let $G \in H[\mathcal{K} \cup \overline{\mathcal{K}}]$ for some connected graph $H$. Then $\dim_{1,f}(G)\ge \dim_f(G)=\frac{n}{2}$ by Observation~\ref{obs_bounds}(d) and Theorem~\ref{n_2}. Since $\dim_{1,f}(G)\le \frac{n}{2}$ by Proposition~\ref{frac_kdim_bounds}, $\dim_{1,f}(G)=\frac{n}{2}$.

($\Rightarrow$) Let $\dim_{1,f}(G)=\frac{n}{2}$. It suffices to show that each twin equivalence class of $V(G)$ has cardinality at least two. Assume, to the contrary, that there exists a twin equivalence class $Q \subset V(G)$ consisting of exactly one element; let $z \in Q$. Let $h:V(G) \rightarrow [0,1]$ be a function defined by $h(z)=0$ and $h(v)=\frac{1}{2}$ for each $v \in V(G)-\{z\}$. Since $|R_1\{z,u\}|\ge3$ for any $u\in V(G)-\{z\}$, $h$ is a $1$-truncated resolving function of $G$ with $h(V(G))=\frac{n-1}{2}$, and hence $\dim_{1,f}(G)\le\frac{n-1}{2}$, a contradiction. So, each twin equivalence class of $V(G)$ must have cardinality at least two. By the connectedness of $G$, we conclude that $G \in H[\mathcal{K} \cup \overline{\mathcal{K}}]$ for some connected graph $H$.~\hfill
\end{proof}

Since $\dim_f(G) \le \dim_{k,f}(G) \le \dim_{1,f}(G)$ by Observation~\ref{obs_bounds}(d), Theorems~\ref{n_2} and~\ref{frac_adim_n2} imply the following

\begin{corollary}
Let $G$ be a connected graph of order $n \ge 2$, and let $k\in\mathbb{Z}^+$. Then $\dim_{k,f}(G)=\frac{n}{2}$ if and only if $G \in H[\mathcal{K} \cup \overline{\mathcal{K}}]$ for some connected graph $H$.
\end{corollary}


\section{$\dim_{k,f}(G)$ of some graph classes}\label{Sec_graphs}

In this section, we examine $\dim_{k,f}(G)$ for some classes of graphs.

We recall some terminology. The \emph{join} of two graphs $G$ and $H$, denoted by $G+H$, is the graph obtained from the disjoint union of $G$ and $H$ by joining an edge between each vertex of $G$ and each vertex of $H$. 

The \emph{degree} of a vertex $v\in V(G)$ is $|N(v)|$; a \emph{leaf} is a vertex of degree one and a \emph{major vertex} is a vertex of degree at least three. Now, fix a tree $T$. A leaf $\ell$ is called a \emph{terminal vertex} of a major vertex $v$ if $d(\ell, v)<d (\ell, w)$ for every other major vertex $w$ in $T$. The \emph{terminal degree}, $ter(v)$, of a major vertex $v$ is the number of terminal vertices of $v$ in $T$, and an \emph{exterior major vertex} is a major vertex that has positive terminal degree. Let $M(T)$ be the set of exterior major vertices of $T$ and let $L(T)$ be the set of leaves of $T$. Let $M_1(T)=\{w \in M(T) : ter(w) =1\}$, $M_2(T)=\{w\in M(T): ter(w) \ge 2\}$; then $M(T)=M_1(T) \cup M_2(T)$. Let $\sigma(T)=|L(T)|$, $ex(T)=|M(T)|$ and $ex_1(T)=|M_1(T)|$. 

 Now, we recall the fractional metric dimension of some classes of graphs. 

\begin{theorem}\label{thm_frac}
\begin{itemize}
\item[(a)] \emph{\cite{fracdim_yi}} For any tree $T$, $\dim_f(T)=\frac{1}{2}(\sigma(T)-ex_1(T))$.
\item[(b)] \emph{\cite{fracdim1}} For the Petersen graph $\mathcal{P}$, $\dim_f(\mathcal{P})=\frac{5}{3}$.
\item[(c)] \emph{\cite{fracdim1}} For $n\ge3$,
$\dim_f(C_n)=\left\{
\begin{array}{ll}
\frac{n}{n-1} & \mbox{if $n$ is odd},\\
\frac{n}{n-2} & \mbox{if $n$ is even}.
\end{array}\right.$
\item[(d)] \emph{\cite{fracdim1}} For the wheel graph $W_n=C_{n-1}+K_1$ of order $n \ge 5$,
$\dim_f(W_n)=\left\{
\begin{array}{ll}
2 & \mbox{if } n=5,\\
\frac{3}{2} & \mbox{if } n=6,\\
\frac{n-1}{4} & \mbox{if } n \ge 7.
\end{array}\right.$
\item[(e)] \emph{\cite{fracdim_yi}} For $m \ge 2$, let $G=K_{a_1,a_2, \ldots, a_m}$ be a complete $m$-partite graph of order $n=\sum_{i=1}^{m}a_i$, and let $s$ be the number of partite sets of $G$ consisting of exactly one element. Then
\begin{equation*}
\dim_f(G)=\left\{
\begin{array}{ll}
\frac{n-1}{2} & \mbox{if }s=1,\\
\frac{n}{2} & \mbox{otherwise}.
\end{array}\right.
\end{equation*}
\item[(f)]  \emph{\cite{fracdim1}} For the grid graph $G=P_s \times P_t$ ($s,t \ge 2$), $\dim_f(G)=2$.
\end{itemize}
\end{theorem}


\subsection{Graphs $G$ with $\diam(G)\le2$}

If $G$ is a graph with $\diam(G)\le 2$, then $\dim_{k,f}(G)=\dim_f(G)$ for any $k\in\mathbb{Z}^+$. So, for any $k\in\mathbb{Z}^+$, we have $\dim_{k,f}(C_n+K_1)=\dim_f(C_n+K_1)$, $\dim_{k,f}(P_n+K_1)=\dim_f(P_n+K_1)$, $\dim_{k,f}(\mathcal{P})=\dim_f(\mathcal{P})$ for the Petersen graph $\mathcal{P}$, and $\dim_{k,f}(G)=\dim_f(G)$ for any complete multipartite graph $G$, for examples. So, Theorem~\ref{thm_frac}(b)(d)(e) implies the following.

\begin{corollary}
Let $k\in\mathbb{Z}^+$.
\begin{itemize}
\item[(a)] For the Petersen graph $\mathcal{P}$, $\dim_{k,f}(\mathcal{P})=\frac{5}{3}$.
\item[(b)]For $n \ge 3$, 
\begin{equation*}
\dim_{k,f}(C_n+K_1)=\left\{
\begin{array}{ll}
2 & \mbox{ if } n \in \{3,4\},\\
\frac{3}{2} & \mbox{ if }n=5,\\
\frac{n}{4} & \mbox{ if }n\ge 6. 
\end{array}\right.
\end{equation*}
\item[(c)] For $m \ge 2$, let $G=K_{a_1, a_2, \ldots, a_m}$ be a complete $m$-partite graph of order $n=\sum_{i=1}^{m}a_i$. Let $s$ be the number of partite sets of $G$ consisting of exactly one element. Then
\begin{equation*}
\dim_{k,f}(G)=\left\{
\begin{array}{ll}
\frac{n-1}{2} & \mbox{if }s=1,\\
\frac{n}{2} & \mbox{otherwise}.
\end{array}\right.
\end{equation*}
\end{itemize}
\end{corollary}

Next, we examine $\dim_{k,f}(P_n+K_1)$.

\begin{theorem}\label{fdim_fan}
Let $k\in\mathbb{Z}^+$. 
\begin{itemize}
\item[(a)] For $n \ge 1$, 
\begin{equation*}
\dim_{k,f}(P_n+K_1)=\dim_f(P_n+K_1)=\left\{
\begin{array}{ll}
\frac{n+1}{2} & \mbox{ if } n \in\{1,2,3\},\\
\frac{5}{3} & \mbox{ if } n \in \{4,5\},\\
\frac{n+1}{4} & \mbox{ if } n \ge 6 \mbox{ and } n\equiv 1,3\pmod4 ,\\
\frac{n+2}{4} & \mbox{ if } n \ge 6 \mbox{ and } n\equiv 2\pmod4.
\end{array}\right.
\end{equation*}
\item[(b)] If $n\ge 8$ and $n\equiv 0 \pmod4$, then $\frac{n}{4} \le \dim_{k,f}(P_n+K_1)=\dim_f(P_n+K_1)\le \frac{n+2}{4}$.
\end{itemize}
\end{theorem}

\begin{proof}
Let $k\in\mathbb{Z}^+$ and $n\ge1$. 
Since $\diam(P_n+K_1) \le 2$, $\dim_{k,f}(P_n+K_1)=\dim_f(P_n+K_1)$ by Observation~\ref{obs_diam}(b). Let $G=P_n+K_1$ such that $P_n$ is given by $u_1, u_2, \ldots, u_n$ and $w$ is the vertex in $K_1$. Let $g: V(G) \rightarrow [0,1]$ be any minimum ($k$-truncated) resolving function of $G$. 

First, let $n\in\{1,2\}$; then $P_n+K_1=K_{n+1}$ and $\dim_{k,f}(P_n+K_1)=\dim_f(P_n+K_1)=\frac{n+1}{2}$ by Theorem~\ref{thm_frac}(e). Second, let $n=3$. Since $u_1$ and $u_3$ are twins and $u_2$ and $w$ are twins in $G$, $g(u_1)+g(u_3) \ge 1$ and $g(u_2)+g(w) \ge 1$ by Observation~\ref{obs_twin}(a)(b); thus $g(V(G))\ge 2$, and hence $\dim_{k,f}(G)=\dim_f(G) \ge 2$. By Theorem~\ref{characterization_mixed}(b), $\dim_{k,f}(G)=\dim_f(G)=2$. 

Third, let $n=4$. Then $R_k\{u_1, u_2\}=R\{u_1, u_2\}=V(G)-\{w,u_4\}$, $R_k\{u_1, u_3\}=R\{u_1, u_3\}=V(G)-\{w, u_2\}$, $R_k\{u_2, u_4\}=R\{u_2, u_4\}=V(G)-\{w,u_3\}$, $R_k\{u_3,u_4\}=R\{u_3,u_4\}=V(G)-\{w,u_1\}$, $R_k\{u_2,w\}=R\{u_2,w\}=V(G)-\{u_1,u_3\}$, and $R_k\{u_3,w\}=R\{u_3,w\}=V(G)-\{u_2,u_4\}$. So, $g(V(G))-g(w)-g(u_4)\ge 1$, $g(V(G))-g(w)-g(u_2)\ge 1$, $g(V(G))-g(w)-g(u_3)\ge 1$, $g(V(G))-g(w)-g(u_1)\ge 1$, $3g(V(G))-3g(u_1)-3g(u_3)\ge 3$ and $3g(V(G))-3g(u_2)-3g(u_4)\ge 3$. By summing over the six inequalities, we have $6g(V(G))\ge 10$; thus, $\dim_{k,f}(G)=\dim_f(G) \ge \frac{5}{3}$. On the other hand, if $h: V(G) \rightarrow [0,1]$ is a function defined by $h(w)=h(u_i)=\frac{1}{3}$ for each $i\in\{1,2,3,4\}$, then $h$ is a ($k$-truncated) resolving function of $G$ with $h(V(G))=\frac{5}{3}$; thus $\dim_{k,f}(G)=\dim_f(G) \le \frac{5}{3}$. Therefore, $\dim_{k,f}(G)=\dim_f(G)=\frac{5}{3}$. 

Fourth, let $n=5$. Then $R_k\{u_1, u_2\}=R\{u_1, u_2\}=V(G)-\{w, u_4, u_5\}$, $R_k\{u_4, u_5\}=R\{u_4, u_5\}=V(G)-\{w, u_1, u_2\}$, $R_k\{u_2, w\}=R\{u_2, w\}=V(G)-\{u_1,u_3\}$, $R_k\{u_3, w\}=R\{u_3, w\}=V(G)-\{u_2, u_4\}$ and $R_k\{u_4, w\}=R\{u_4, w\}=V(G)-\{u_3, u_5\}$. So, $g(V(G))-g(w)-g(u_4)-g(u_5)\ge 1$, $g(V(G))-g(w)-g(u_1)-g(u_2)\ge 1$, $g(V(G))-g(u_1)-g(u_3)\ge 1$, $g(V(G))-g(u_2)-g(u_4)\ge 1$ and $g(V(G))-g(u_3)-g(u_5)\ge 1$. By summing over the five inequalities, we have $3g(V(G)) \ge 5$; thus, $\dim_{k,f}(G)=\dim_f(G) \ge \frac{5}{3}$. On the other hand, if $h: V(G) \rightarrow [0,1]$ is a function defined by $h(w)=0$ and $h(u_i)=\frac{1}{3}$ for each $i\in\{1,2,3,4,5\}$, then $h$ is a ($k$-truncated) resolving function of $G$ with $h(V(G))=\frac{5}{3}$; thus $\dim_{k,f}(G)=\dim_f(G) \le \frac{5}{3}$. Therefore, $\dim_{k,f}(G)=\dim_f(G)=\frac{5}{3}$. 

Next, let $n \ge 6$, and we consider the following cases.

\emph{Case 1: $n\equiv 1,3 \pmod4$.} In this case, $n$ is odd; let $n=2x+1$, where $x \ge 3$. Note that $R_k\{u_1, u_2\}=R\{u_1, u_2\}=\{u_1,u_2,u_3\}$, $R_k\{u_{2x}, u_{2x+1}\}=R\{u_{2x}, u_{2x+1}\}=\{u_{2x-1},u_{2x},u_{2x+1}\}$, and $R_k\{u_{i-1}, u_{i+1}\}=R\{u_{i-1}, u_{i+1}\}=\{u_{i-2}, u_{i-1}, u_{i+1}, u_{i+2}\}$ for each odd $i\in\{3,5,\ldots, 2x-1\}$. So, $g(u_1)+g(u_2)+g(u_3) \ge 1$, $g(u_{2x-1})+g(u_{2x})+g(u_{2x+1}) \ge 1$, and $g(u_{i-2})+g(u_{i-1})+g(u_{i+1})+g(u_{i+2}) \ge 1$ for each odd $i\in\{3,5,\ldots, 2x-1\}$. By summing over the $x+1$ inequalities, we have $2g(V(G))\ge 2\sum_{i=1}^{2x+1}g(u_i)\ge x+1$, and thus $\dim_{k,f}(G)=\dim_f(G) \ge \frac{x+1}{2}=\frac{n+1}{4}$. Now, let $h:V(G) \rightarrow[0,1]$ be a function defined by $h(u_i)=\frac{1}{2}$ if $i$ is odd, and $h(w)=h(u_i)=0$ if $i$ is even. Then $h$ is a ($k$-truncated) resolving function of $G$ with $h(V(G))=\frac{x+1}{2}$ since $|R\{u_i,u_j\} \cap (\cup_{\alpha=1}^{x+1}\{u_{2\alpha-1}\})| \ge 2$ and $|R\{u_i,w\} \cap (\cup_{\alpha=1}^{x+1}\{u_{2\alpha-1}\})| \ge 2$. So, $\dim_{k,f}(G)=\dim_f(G) \le \frac{x+1}{2}=\frac{n+1}{4}$. Thus, $\dim_{k,f}(G)=\dim_f(G) = \frac{n+1}{4}$ for $n\equiv 1,3 \pmod4$.

\emph{Case 2: $n\equiv 2 \pmod4$.} Let $n=4x+2$, where $x\ge 1$. Since $R_k\{u_1, u_2\}=R\{u_1, u_2\}=\{u_1, u_2, u_3\}$, $R_k\{u_{4x+1}, u_{4x+2}\}=R\{u_{4x+1}, u_{4x+2}\}=\{u_{4x}, u_{4x+1}, u_{4x+2}\}$, and $R_k\{u_{4i+1}, u_{4i+2}\}=R\{u_{4i+1}, u_{4i+2}\}=\{u_{4i}, u_{4i+1}, u_{4i+2}, u_{4i+3}\}$ for each $i\in\{1, 2, \ldots, x-1\}$, we have $g(u_1)+g(u_2)+g(u_3) \ge 1$, $g(u_{4x})+g(u_{4x+1})+g(u_{4x+2}) \ge 1$, and $g(u_{4i})+g(u_{4i+1})+g(u_{4i+2})+g(u_{4i+3})\ge 1$ for each $i\in\{1, 2, \ldots, x-1\}$. By summing over the $x+1$ inequalities, we have $g(V(G)) \ge \sum_{i=1}^{4x+2}g(u_i) \ge x+1=\frac{n+2}{4}$; thus, $\dim_{k,f}(G)=\dim_f(G) \ge \frac{n+2}{4}$. If we let $h:V(G) \rightarrow [0,1]$ be a function defined by $h(w)=0$, $h(u_1)=h(u_{4x+2})=\frac{1}{2}$, and $h(u_j)=\frac{1}{4}$ for each $j \in \{2,3, \ldots, 4x+1\}$, then $h$ is a ($k$-truncated) resolving function of $G$ with $h(V(G))=\frac{n+2}{4}$; thus $\dim_{k,f}(G)=\dim_f(G) \le \frac{n+2}{4}$. Therefore, $\dim_{k,f}(G)=\dim_f(G)=\frac{n+2}{4}$ for $n\equiv 2 \pmod4$.

\emph{Case 3: $n\equiv 0 \pmod4$.} Let $n=4x$, where $x \ge 2$. For each $i\in\{0,1,\ldots, x-1\}$, $R_k\{u_{4i+2}, u_{4i+3}\}=R\{u_{4i+2}, u_{4i+3}\}=\{u_{4i+1}, u_{4i+2}, u_{4i+3}, u_{4i+4}\}$ and $g(u_{4i+1})+g(u_{4i+2})+g(u_{4i+3})+g(u_{4i+4})\ge 1$. By summing over the $x$ inequalities, we have $g(V(G))\ge \sum_{i=1}^{4x}g(u_i)\ge x$; thus, $\dim_{k,f}(G)=\dim_f(G)\ge \frac{n}{4}$. If we let $h:V(G)\rightarrow [0,1]$ be a function defined by $h(w)=0$, $h(u_1)=\frac{1}{2}=h(u_{4x})$, and $h(u_i)=\frac{1}{4}$ for each $i\in\{2,3,\ldots, 4x-1\}$, then $h$ is a ($k$-truncated) resolving function of $G$ with $g(V(G))=\frac{n+2}{4}$; thus, $\dim_{k,f}(G)=\dim_f(G)\le \frac{n+2}{4}$. So, $\frac{n}{4} \le \dim_{k,f}(G)=\dim_f(G)\le \frac{n+2}{4}$ for $n\equiv 0 \pmod4$.~\hfill
\end{proof}


\subsection{Cycles}

We determine $\dim_{k,f}(C_n)$ for any $k\in\mathbb{Z}^+$ and for $n\ge3$.

\begin{theorem}\label{fkdim_cycle}
For any $k\in\mathbb{Z}^+$ and for $n\ge 3$, 
\begin{equation*}
\dim_{k,f}(C_n)=\left\{
\begin{array}{ll}
\frac{n}{n-1} & \mbox{ if } n \le 2k+3 \mbox{ and $n$ is odd},\\
\frac{n}{n-2} & \mbox{ if } n\le 2k+3 \mbox{ and $n$ is even},\\
\frac{n}{2(k+1)} & \mbox{ if } n\ge 2k+4.
\end{array}\right.
\end{equation*}
\end{theorem}

\begin{proof}
Let $k\in\mathbb{Z}^+$. For $n \ge 3$, let $C_n$ be given by $u_0,u_1, \ldots, u_{n-1}, u_0$. Let $g:V(C_n)\rightarrow [0,1]$ be any minimum $k$-truncated resolving function of $C_n$. 

First, let $n \le 2k+3$; then $\diam(C_n) \le \lfloor \frac{2k+3}{2}\rfloor=k+1$. By Observation~\ref{obs_diam}(b), $\dim_{k, f}(C_n)=\dim_f(C_n)$. So, by Theorem~\ref{thm_frac}(c), $\dim_{k, f}(C_n)=\frac{n}{n-1}$ for an odd $n$ and $\dim_{k, f}(C_n)=\frac{n}{n-2}$ for an even $n$.

Second, let $n\ge 2k+4$. Note that, for each $i \in \{0,1,\ldots, n-1\}$, $R_k\{u_{i}, u_{i+2}\}=\cup_{j=0}^{k}\{u_{i-j}, u_{i+2+j}\}$, where the subscript it taken modulo $n$; thus $\sum_{j=0}^{k}(g(u_{i-j})+g(u_{i+2+j})) \ge 1$. By summing over $n$ such inequalities, we have $2(k+1) g(V(C_n)) \ge n$ since each vertex appears $2(k+1)$ times in the $n$ inequalities. So, $g(V(C_n)) \ge \frac{n}{2(k+1)}$, and hence $\dim_{k,f}(C_n) \ge \frac{n}{2(k+1)}$. On the other hand, if we let $h:V(C_n)\rightarrow [0,1]$ be a function defined by $h(u_i)=\frac{1}{2(k+1)}$ for each $i \in \{0,1,\ldots,n-1\}$, then $h$ is a $k$-truncated resolving function of $C_n$ with $h(V(C_n))=\frac{n}{2(k+1)}$. To see this, for any distinct $x,y\in \{0, 1, \ldots, n-1\}$, note that $|R_k\{u_x, u_y\}| \ge 2(k+1)$ and $h(R_k\{u_x, u_y\}) \ge \frac{1}{2(k+1)}\cdot 2(k+1)=1$. So, $\dim_{k, f}(C_n) \le h(V(C_n))=\frac{n}{2(k+1)}$. Therefore, $\dim_{k,f}(C_n)=\frac{n}{2(k+1)}$.~\hfill
\end{proof}

Theorem~\ref{fkdim_cycle} implies the following

\begin{corollary}\label{fadim_cycle}
For $n \ge 3$, 
\begin{equation*}
\dim_{1,f}(C_n)=\left\{
\begin{array}{ll}
\frac{n}{2} & \mbox{ if } n \in \{3,4\},\\ 
\frac{n}{4} & \mbox{ if } n \ge 5. 
\end{array}\right.
\end{equation*}
\end{corollary}


\subsection{Grid graphs}

We examine $\dim_{k,f}(P_s\times P_t)$ for $s,t\ge2$. We show that, for some $k>1$, $\frac{\dim_{1,f}(G)}{\dim_{k,f}(G)}$ and $\frac{\dim_{k,f}(G)}{\dim_f(G)}$ can be arbitrarily large. We also characterize grid graphs $G$ satisfying $\dim_{1,f}(G)=\dim_f(G)$.

We recall some notations. Let $f(x)$ and $g(x)$ be functions defined on some subset of real numbers. We write $f(x)=O(g(x))$ if there exist positive constants $N$ and $C$ such that $|f(x)| \le C |g(x)|$ for all $x >N$, $f(x)=\Omega(g(x))$ if $g(x)=O(f(x))$, and $f(x)=\Theta(g(x))$ if $f(x)=O(g(x))$ and $f(x)=\Omega(g(x))$. We note that $\frac{\dim_1(G)}{\dim(G)}$ can be arbitrarily large (see~\cite{broadcast}) and that both $\frac{\dim_1(G)}{\dim_k(G)}$ and $\frac{\dim_k(G)}{\dim(G)}$ can be arbitrarily large for some $k>1$ (see~\cite{distKdim}). 

First, we show that $\frac{\dim_{1,f}(G)}{\dim_{k,f}(G)}$ and $\frac{\dim_{k,f}(G)}{\dim_f(G)}$ can be arbitrarily large for some $k>1$.

\begin{proposition}\emph{\cite{broadcast}}\label{adim_theta}
If $G=P_m \times P_m$ for $m\ge 2$, then $\dim_1(G)=\Theta(m^2)$. 
\end{proposition}

\begin{proposition}\emph{\cite{distKdim}}\label{kdim_theta}
For some positive integer $k>1$, let $G=P_{k^2} \times P_{k^2}$. Then $\dim_k(G)=\Theta(k^2)$.
\end{proposition}

\begin{proposition}\label{fadim_theta}
If $G=P_{4m} \times P_{3m}$ for $m\ge 1$, then $\dim_{1,f}(G)=\Theta(m^2)$. 
\end{proposition}

\begin{proof}
By Proposition~\ref{adim_theta} and Observation~\ref{obs_bounds}(d), $\dim_{1,f}(G)=O(m^2)$. To see that $\dim_{1,f}(G)=\Omega(m^2)$, suppose that the grid graph $G=P_{4m} \times P_{3m}$ is drawn in the $xy$-plane with the four corners at $(1,1)$, $(4m,1)$, $(1,3m)$ and $(4m,3m)$ with horizontal and vertical edges of equal lengths, and let $g:V(G)\rightarrow [0,1]$ be any $1$-truncated resolving function of $G$. Then, for every $P_4 \times P_3$ subgraph, say $B_{i,j}$, of $G$ with the four corners $(1+4i, 1+3j)$, $(4+4i, 1+3j)$, $(1+4i,3+3j)$ and $(4+4i, 3+3j)$, where $i,j \in \{0,1,\ldots, m-1\}$, we have $R_1\{(2+4i, 2+3j), (3+4i, 2+3j)\} \subset V(B_{i,j})$, and thus $g(V(B_{i,j})) \ge 1$. So, $\dim_{1,f}(G) \ge \sum_{j=0}^{m-1}\sum_{i=0}^{m-1} g(V(B_{i,j})) \ge m^2$, and hence $\dim_{1,f}(G)=\Omega(m^2)$. Therefore, $\dim_{1,f}(G)=\Theta(m^2)$.~\hfill
\end{proof}

\begin{theorem}\label{comp_fkdim}
For some positive integer $k>1$, let $G=P_{(2k+2)^2} \times P_{(2k+1)^2}$. Then $\dim_{k,f}(G)=\Theta(k^2)$, and thus both $\frac{\dim_{1,f}(G)}{\dim_{k,f}(G)}$ and $\frac{\dim_{k,f}(G)}{\dim_f(G)}$ can be arbitrarily large.
\end{theorem}

\begin{proof}
For some positive integer $k>1$, let $G=P_{(2k+2)^2} \times P_{(2k+1)^2}$. Then $\dim_f(G)=2$ by Theorem~\ref{thm_frac}(f), and $\dim_{1,f}(G)=\Theta(k^4)$ by Proposition~\ref{fadim_theta}. Next, we show that $\dim_{k,f}(G)=\Theta(k^2)$. By Proposition~\ref{kdim_theta} and Observation~\ref{obs_bounds}(c), $\dim_{k,f}(G)=O(k^2)$. To see that $\dim_{k,f}(G)=\Omega(k^2)$, notice that $G$ contains disjoint union of $(2k+2)(2k+1)$ copies of $P_{2k+2} \times P_{2k+1}$. Let $g:V(G)\rightarrow [0,1]$ be any $k$-truncated resolving function of $G$. For each subgraph $P_{2k+2} \times P_{2k+1}$ of $G$, if $x$ and $y$ are the two adjacent central vertices of $P_{2k+2} \times P_{2k+1}$, then $R_k\{x,y\} \subseteq V(P_{2k+2} \times P_{2k+1})$, and hence $g(V(P_{2k+2} \times P_{2k+1})) \ge g(R_k\{x,y\}) \ge 1$; thus $\dim_{k,f}(G) \ge (2k+2)(2k+1)$. So, $\dim_{k,f}(G)=\Omega(k^2)$. Therefore, $\dim_{k,f}(G)=\Theta(k^2)$ for some positive integer $k>1$, and both $\frac{\dim_{1,f}(G)}{\dim_{k,f}(G)}$ and $\frac{\dim_{k,f}(G)}{\dim_f(G)}$ can be arbitrarily large.~\hfill
\end{proof}

Next, we characterize grid graphs $G$ satisfying $\dim_{1,f}(G)=\dim_f(G)$.

\begin{proposition}
For the grid graph $G=P_s \times P_t$ with $s\ge t\ge2$, $\dim_{1,f}(G)=\dim_f(G)$ if and only if $G\in\{P_2 \times P_2, P_3\times P_2, P_4\times P_2, P_3 \times P_3\}$.
\end{proposition}

\begin{figure}[ht]
\centering
\begin{tikzpicture}[scale=.7, transform shape]

\node [draw, fill=black, shape=circle, scale=.8] (a1) at  (0,0) {};
\node [draw, fill=black, shape=circle, scale=.8] (a2) at  (1,0) {};
\node [draw, fill=black, shape=circle, scale=.8] (a3) at  (0,1) {};
\node [draw, fill=black, shape=circle, scale=.8] (a4) at  (1,1) {};

\node [draw, fill=black, shape=circle, scale=.8] (b1) at  (3,0) {};
\node [draw, shape=circle, scale=.8] (b2) at  (4, 0) {};
\node [draw, fill=black, shape=circle, scale=.8] (b3) at  (5,0) {};
\node [draw, fill=black, shape=circle, scale=.8] (b4) at  (3,1) {};
\node [draw, shape=circle, scale=.8] (b5) at  (4,1) {};
\node [draw, fill=black, shape=circle, scale=.8] (b6) at  (5,1) {};

\node [draw, shape=circle, scale=.8] (c1) at  (7,0) {};
\node [draw, fill=black, shape=circle, scale=.8] (c2) at  (8, 0) {};
\node [draw, fill=black, shape=circle, scale=.8] (c3) at  (9,0) {};
\node [draw, shape=circle, scale=.8] (c4) at  (10,0) {};
\node [draw, shape=circle, scale=.8] (c5) at  (7,1) {};
\node [draw, fill=black, shape=circle, scale=.8] (c6) at  (8,1) {};
\node [draw, fill=black, shape=circle, scale=.8] (c7) at  (9,1) {};
\node [draw, shape=circle, scale=.8] (c8) at  (10,1) {};

\node [draw, shape=circle, scale=.8] (d1) at  (12,-0.5) {};
\node [draw, fill=black, shape=circle, scale=.8] (d2) at  (13,-0.5) {};
\node [draw, shape=circle, scale=.8] (d3) at  (14,-0.5) {};
\node [draw, fill=black, shape=circle, scale=.8] (d4) at  (12,0.5) {};
\node [draw, shape=circle, scale=.8] (d5) at  (13,0.5) {};
\node [draw, fill=black, shape=circle, scale=.8] (d6) at  (14,0.5) {};
\node [draw, shape=circle, scale=.8] (d7) at  (12,1.5) {};
\node [draw, fill=black, shape=circle, scale=.8] (d8) at  (13,1.5) {};
\node [draw, shape=circle, scale=.8] (d9) at  (14,1.5) {};

\draw(a1)--(a2)--(a4)--(a3)--(a1);
\draw(b1)--(b2)--(b3)--(b6)--(b5)--(b4)--(b1);\draw(b2)--(b5);
\draw(c1)--(c2)--(c3)--(c4)--(c8)--(c7)--(c6)--(c5)--(c1);\draw(c2)--(c6);\draw(c3)--(c7);
\draw(d1)--(d2)--(d3)--(d6)--(d9)--(d8)--(d7)--(d4)--(d1);\draw(d4)--(d5)--(d6);\draw(d2)--(d5)--(d8);

\end{tikzpicture}
\caption{The four grid graphs $G$ with $\dim_{1,f}(G)=\dim_f(G)$.}\label{fig_grid_characterization}
\end{figure}
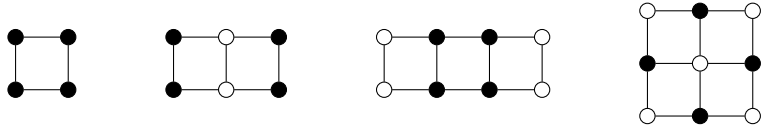

\begin{proof}
Let $G=P_s \times P_t$ with $s\ge t\ge2$.

($\Leftarrow$) Let $G\in\{P_2 \times P_2, P_3\times P_2, P_4\times P_2, P_3 \times P_3\}$; see Figure~\ref{fig_grid_characterization}. For each grid graph $G$ in Figure~\ref{fig_grid_characterization}, if we let $g: V(G)\rightarrow [0,1]$ be a function defined by $g(v)=\frac{1}{2}$ if $v$ is a solid vertex and $g(w)=0$ if $w$ is an open vertex, then $g$ is a $1$-truncated resolving function of $G$. So, $\dim_{1,f}(G)\le g(V(G))=2$. Since $\dim_{1,f}(G)\ge \dim_f(G)=2$ by Observation~\ref{obs_bounds}(d) and Theorem~\ref{thm_frac}(f), $\dim_{1,f}(G)=\dim_f(G)=2$.

($\Rightarrow$) Let $\dim_{1,f}(G)=\dim_f(G)$. Let the grid graph $G$ be drawn in the $xy$-plane with the four corners at $(1,1)$, $(s,1)$, $(1,t)$ and $(s,t)$ with horizontal and vertical edges of equal lengths; we denote by $u_{i,j}$ the vertex $(i,j)$. Let $g:V(G)\rightarrow [0,1]$ be any $1$-truncated resolving function of $G$. 

First, suppose $G=P_s \times P_t$ for $s,t\ge4$. Then $R_1\{u_{1,1}, u_{2,2}\}=\{u_{1,1}, u_{2,2}, u_{2,3}, u_{3,2}\}$, $R_1\{u_{1,2}, u_{2,1}\}=\{u_{1,2}, u_{1,3}, u_{2,1}, u_{3,1}\}$ and $R_1\{u_{s-1,t}, u_{s,t-1}\}=\{u_{s-2,t}, u_{s-1,t}, u_{s,t-1}, u_{s,t-2}\}$; thus $g(u_{1,1})+g(u_{2,2})+g(u_{2,3})+g(u_{3,2}) \ge1$, $g(u_{1,2})+g(u_{1,3})+g(u_{2,1})+g(u_{3,1}) \ge1$ and $g(u_{s-2,t})+g(u_{s-1,t})+g(u_{s,t-1})+g(u_{s,t-2}) \ge1$. By summing over the three inequalities, we have $g(V(G))\ge 3$ since $s,t\ge4$; thus $\dim_{1,f}(G)\ge3>2=\dim_f(G)$. So, $\dim_{1,f}(G)=\dim_f(G)$ implies $t\le 3$.

Second, suppose $G=P_s \times P_3$ for $s\ge 3$. If $s=3$, then $\dim_{1,f}(G)=2=\dim_f(G)$ as shown above. If $s=4$, then $R_1\{u_{1,1},u_{2,2}\}=\{u_{1,1},u_{2,2},u_{2,3},u_{3,2}\}$, $R_1\{u_{1,2},u_{2,1}\}=\{u_{1,2},u_{1,3},u_{2,1},u_{3,1}\}$, $R_1\{u_{3,3},u_{4,2}\}=\{u_{2,3},u_{3,3},u_{4,2},u_{4,1}\}$, $R_1\{u_{3,2},u_{4,3}\}=\{u_{2,2},u_{3,1},u_{3,2},u_{4,3}\}$ and $R_1\{u_{1,1}, u_{4,3}\}=\{u_{1,1},u_{1,2},u_{2,1},u_{3,3},u_{4,2},u_{4,3}\}$; thus, $g(u_{1,1})+g(u_{2,2})+g(u_{2,3})+g(u_{3,2})\ge1$, $g(u_{1,2})+g(u_{1,3})+g(u_{2,1})+g(u_{3,1})\ge1$, $g(u_{2,3})+g(u_{3,3})+g(u_{4,2})+g(u_{4,1})\ge1$, $g(u_{2,2})+g(u_{3,1})+g(u_{3,2})+g(u_{4,3})\ge1$ and $g(u_{1,1})+g(u_{1,2})+g(u_{2,1})+g(u_{3,3})+g(u_{4,2})+g(u_{4,3})\ge1$. By summing over the five inequalities, we have $2g(V(G))\ge 5$, and hence $\dim_{1,f}(G)\ge \frac{5}{2}>2=\dim_f(G)$. If $s\ge 5$, then $R_1\{u_{1,1}, u_{2,2}\}=\{u_{1,1}, u_{2,2}, u_{2,3}, u_{3,2}\}$, $R_1\{u_{1,2}, u_{2,1}\}=\{u_{1,2}, u_{1,3}, u_{2,1}, u_{3,1}\}$ and $R_1\{u_{s-1,3}, u_{s,2}\}=\{u_{s-2,3}, u_{s-1,3}, u_{s,2}, u_{s,1}\}$; thus, $g(u_{1,1})+g(u_{2,2})+g(u_{2,3})+g(u_{3,2})\ge1$, $g(u_{1,2})+g(u_{1,3})+g(u_{2,1})+g(u_{3,1})\ge1$ and $g(u_{s-2,3})+g(u_{s-1,3})+g(u_{s,2})+g(u_{s,1})\ge1$. By summing over the three inequalities, we have $g(V(G))\ge 3$ since $s\ge5$, and hence $\dim_{1,f}(G)\ge3>2=\dim_f(G)$. 

Third, suppose $G=P_s \times P_2$ for $s\ge 2$. If $s\in\{2,3,4\}$, then $\dim_{1,f}(G)=2=\dim_f(G)$ as shown above. So, suppose $s\ge 5$. Then $R_1\{u_{1,1}, u_{2,2}\}=\{u_{1,1}, u_{2,2}, u_{3,2}\}$, $R_1\{u_{1,2},u_{2,1}\}=\{u_{1,2},u_{2,1},u_{3,1}\}$ and $R_1\{u_{s,1},u_{s,2}\}=\{u_{s-1,1},u_{s-1,2},u_{s,1},u_{s,2}\}$; thus, $g(u_{1,1})+g(u_{2,2})+g(u_{3,2})\ge 1$, $g(u_{1,2})+g(u_{2,1})+g(u_{3,1})\ge1$ and $g(u_{s-1,1})+g(u_{s-1,2})+g(u_{s,1})+g(u_{s,2})\ge1$. By summing over the three inequalities, we have $g(V(G))\ge 3$ since $s\ge5$, and thus $\dim_{1,f}(G)\ge 3>2=\dim_f(G)$.~\hfill
\end{proof}


\subsection{Trees}

We examine $\dim_{k, f}(T)$ for non-trivial trees $T$. For $n\ge 2$, $\dim_{k,f}(P_n)=\dim_f(P_n)$ if and only if $n\in\{2,3,\ldots, k+2\}$ by Theorems~\ref{characterization_mixed}(c) and~\ref{frac_kdim=1}. We characterize trees $T$ satisfying $\dim_{1,f}(T)=\dim_f(T)$. For trees $T$ with $ex(T)=1$ and for any $k\in\mathbb{Z}^+$, we also characterize $T$ satisfying $\dim_{k,f}(T)=\dim_f(T)$. 

We first consider $\dim_{k,f}(P_n)$. 

\begin{theorem}\label{frac_distK_path}
Let $k\in\mathbb{Z}^+$ and $n \ge 2$. 
\begin{itemize}
\item[(a)] If $n\le k+2$, then $\dim_{k,f}(P_n)=1$.
\item[(b)] If $k+3\le n\le 2k+3$, then $\dim_{k,f}(P_n)=\frac{6+2k-n}{5+2k-n}$.
\item[(c)] Let $n\ge 2k+4$.
\begin{itemize}
\item[(i)] If $n \equiv 1 \pmod{(2k+2)}$, then $\dim_{k,f}(P_n)=\frac{n+k}{2k+2}$.
\item[(ii)] If $n \equiv 2, 3, \ldots, k+2 \pmod{(2k+2)}$, then $\dim_{k,f}(P_n)=\lceil\frac{n}{2k+2}\rceil$.
\item[(iii)] If $n\equiv 0 \pmod{(2k+2)} $ or $n\equiv k+3, k+4, \ldots, 2k+1 \pmod{(2k+2)}$, then $\lceil\frac{n}{2k+2}\rceil \le \dim_{k,f}(P_n) \le \lceil\frac{n}{2k+2}\rceil +\frac{1}{2}$.
\end{itemize}
\end{itemize}
\end{theorem}

\begin{proof}
Let $k\in\mathbb{Z}^+$. For $n \ge 2$, let $P_n$ be given by $u_1, u_2, \ldots, u_n$. Let $g: V(P_n) \rightarrow [0,1]$ be any $k$-truncated resolving function of $P_n$. For (a), suppose $n \le k+2$. Then $\dim_{k, f}(P_n)=1$ by Theorem~\ref{frac_kdim=1}. 

For (b), suppose $k+3\le n\le 2k+3$. Then $R_k\{u_1, u_2\}=\cup_{i=1}^{k+2}\{u_i\}$, $R_k\{u_{n-1}, u_n\}=\cup_{i=n-(k+1)}^{n}\{u_i\}$, and $R_k\{u_{i-1}, u_{i+1}\}=V(P_n)-\{u_i\}$ for $u_i \in R_k\{u_1, u_2\} \cap R_k\{u_{n-1}, u_n\}$ (i.e., $n-k-1\le i \le k+2$). By summing over the $(6+2k-n)$ inequalities, we have $(5+2k-n)g(V(P_n)) \ge 6+2k-n$, i.e., $g(V(P_n)) \ge \frac{6+2k-n}{5+2k-n}$; thus, $\dim_{k,f}(P_n) \ge \frac{6+2k-n}{5+2k-n}$. On the other hand, let $h:V(P_n) \rightarrow [0,1]$ be a function defined by
$$h(u_i)=\left\{
\begin{array}{ll}
\frac{1}{5+2k-n} & \mbox{ if }  u_i \in \{u_1, u_n\} \cup (R_k\{u_1, u_2\} \cap R_k\{u_{n-1}, u_n\}),\\
0 & \mbox{ otherwise}.\\
\end{array}\right.$$
Then $h$ is a $k$-truncated resolving function of $P_n$ with $h(V(P_n))=\frac{6+2k-n}{5+2k-n}$; thus, $\dim_{k,f}(P_n) \le \frac{6+2k-n}{5+2k-n}$. Therefore, $\dim_{k,f}(P_n)=\frac{6+2k-n}{5+2k-n}$ for $k+3 \le n \le 2k+3$.

For (c), let $n \ge 2k+4$, and we consider three cases.

Case 1: $n \equiv 1 \pmod{(2k+2)}$. Note that $R_k\{u_1, u_2\}=\cup_{i=1}^{k+2}\{u_i\}$, $R_k\{u_{n-1}, u_n\}=\cup_{i=n-(k+1)}^{n}\{u_i\}$, and $R_k\{u_{i-1}, u_{i+1}\}=\cup_{j=1}^{k+1}\{u_{i-j}, u_{i+j}\}$ for $i \equiv 1 \pmod{(k+1)}$ and $i \not\in\{1,n\}$. By summing over the $\lceil\frac{n}{k+1}\rceil$ inequalities, we have $2g(V(P_n)) \ge \lceil\frac{n}{k+1}\rceil=\frac{n-1}{k+1}+1=\frac{n+k}{k+1}$, i.e., $g(V(P_n)) \ge \frac{n+k}{2(k+1)}$; thus, $\dim_{k,f}(P_n) \ge \frac{n+k}{2k+2}$. On the other hand, if we let $h: V(P_n) \rightarrow [0,1]$ be a function defined by 
$$h(u_i)=\left\{
\begin{array}{ll}
\frac{1}{2} & \mbox{ if }  i \equiv 1 \pmod{(k+1)},\\
0 & \mbox{ otherwise},\\
\end{array}\right.$$
then $h$ is a $k$-truncated resolving function of $P_n$ with $h(V(P_n))=\frac{1}{2}\lceil\frac{n}{k+1}\rceil=\frac{n+k}{2k+2}$; thus, $\dim_{k,f}(P_n) \le \frac{n+k}{2k+2}$. Therefore, $\dim_{k,f}(P_n)=\frac{n+k}{2k+2}$ for $n \equiv 1 \pmod{(2k+2)}$.

Case 2: $n \equiv 2,3,\ldots, k+2 \pmod{(2k+2)}$. Note that $R_k\{u_1, u_2\}=\cup_{i=1}^{k+2}\{u_i\}$, $R_k\{u_{n-1}, u_n\}=\cup_{i=n-(k+1)}^{n}\{u_i\}$, and $R_k\{u_{(2k+2)i+1}, u_{(2k+2)i+2}\}=\cup_{j=1-k}^{k+2} \{u_{(2k+2)i+j}\}$ for $i \in \{1,2,\ldots, \lceil\frac{n}{2k+2}\rceil-2\}$. By summing over the $\lceil\frac{n}{2k+2}\rceil$ inequalities, we have $g(V(P_n)) \ge \lceil\frac{n}{2k+2}\rceil$; thus, $\dim_{k,f}(P_n) \ge \lceil\frac{n}{2k+2}\rceil$. On the other hand, let $h: V(P_n) \rightarrow [0,1]$ be a function defined by 
$$h(u_i)=\left\{
\begin{array}{ll}
\frac{1}{2} & \mbox{ if }  i \equiv 1 \pmod{(k+1)} \mbox{ and } 1 \le i \le \frac{n}{2}, \mbox{ or }n+1-i \equiv 1 \pmod{(k+1)} \mbox{ and } \frac{n}{2} < i \le n,\\
0 & \mbox{ otherwise}.\\
\end{array}\right.$$
Then $h$ is a $k$-truncated resolving function of $P_n$ with $h(V(P_n))=\lceil\frac{n}{2k+2}\rceil$; thus, $\dim_{k,f}(P_n) \le \lceil\frac{n}{2k+2}\rceil$. Therefore, $\dim_{k,f}(P_n)=\lceil\frac{n}{2k+2}\rceil$ for $n \equiv 2,3,\ldots, k+2 \pmod{(2k+2)}$.

Case 3: $n \equiv 0 \pmod{(2k+2)}$ or $n\equiv k+3,\ldots, 2k+1 \pmod{(2k+2)}$. Note that $R_k\{u_{(2k+2)i+k+1}, u_{(2k+2)i+k+2}\}=\cup_{j=1}^{2k+2}\{u_{(2k+2)i+j}\}$ for each $i\in\{0,1,\ldots, \lceil\frac{n}{2k+2}\rceil-2\}$ and $R_k\{u_{n-1}, u_n\}=\cup_{i=n-(k+1)}^{n}\{u_i\}$. By summing over the $\lceil\frac{n}{2k+2}\rceil$ inequalities, we have $g(V(P_n)) \ge \lceil\frac{n}{2k+2}\rceil$; thus, $\dim_{k,f}(P_n) \ge \lceil\frac{n}{2k+2}\rceil$. If $h: V(P_n) \rightarrow [0,1]$ is a function defined by 
$$h(u_i)=\left\{
\begin{array}{ll}
\frac{1}{2} & \mbox{ if } i=\lceil\frac{n}{2}\rceil, \mbox{ or}\\
{} & \ \ \ 1 \le i < \lceil\frac{n}{2}\rceil \mbox{ and } i \equiv 1 \!\!\!\pmod{(k+1)},\mbox{ or }\\
{} & \ \ \ \lceil\frac{n}{2}\rceil < i \le n \mbox{ and } n+1-i \equiv 1 \!\!\!\pmod{(k+1)},\\
0 & \mbox{ otherwise},\\
\end{array}\right.$$
then $h$ is a $k$-truncated resolving function of $P_n$ with $h(V(P_n))=\lceil\frac{n}{2k+2}\rceil+\frac{1}{2}$; thus, $\dim_{k,f}(P_n) \le \lceil\frac{n}{2k+2}\rceil+\frac{1}{2}$. Therefore, $\lceil\frac{n}{2k+2}\rceil\le \dim_{k,f}(P_n)=\lceil\frac{n}{2k+2}\rceil+\frac{1}{2}$ for $n \equiv 0\pmod{(2k+2)}$ or $n\equiv k+3,\ldots, 2k+1\pmod{(2k+2)}$.~\hfill
\end{proof}

Theorem~\ref{frac_distK_path} implies the following.

\begin{corollary}\label{prop_adim_path}
Let $n \ge 2$. 
\begin{itemize}
\item[(a)] If $n\in\{2,3\}$, then $\dim_{1,f}(P_n)=1$; if $n\in\{4,5\}$, then $\dim_{1,f}(P_n)=\frac{8-n}{7-n}$.
\item[(b)] If $n\ge 6$, then
\begin{equation*}
\dim_{1,f}(P_n)=\left\{
\begin{array}{ll}
\frac{n+1}{4} & \mbox{ if } n\equiv1 \pmod4,\\
\lceil\frac{n}{4}\rceil & \mbox{ if } n\equiv2,3\pmod4.
\end{array}\right.
\end{equation*}
\item[(c)] If $n\ge 8$ and $n\equiv0 \pmod4$, then $\frac{n}{4}\le\dim_{1,f}(P_n)\le\frac{n+2}{4}$.
\end{itemize}
\end{corollary}

Next, we characterize non-trivial trees $T$ satisfying $\dim_{1,f}(T)=\dim_f(T)$. 
We recall some terminology and notation. An \emph{exterior degree-two vertex} is a vertex of degree $2$ that lies on a path from a terminal vertex to its major vertex, and an \emph{interior degree-two vertex} is a vertex of degree $2$ such that the shortest path to any terminal vertex includes a major vertex. For $v \in M(T)$, let $T_v$ be the subtree of $T$ induced by $v$ and all vertices belonging to the paths joining $v$ with its terminal vertices. We begin with the following useful lemmas.

\begin{lemma}\label{obs_tree1}
Let $T$ be a tree with $ex(T)\ge1$, and let $g:V(T)\rightarrow[0,1]$ be any minimum resolving function of $T$. Then, for each $v\in M_2(T)$ with $ter(v)=\alpha\ge2$, $g(V(T_v)-\{v\})=\frac{\alpha}{2}$ and $\sum_{v\in M_2(T)}(g(V(T_v))-g(v))=\dim_f(T)$.
\end{lemma}

\begin{proof}
Let $T$ be a tree with $ex(T)\ge1$, and let $g:V(T)\rightarrow[0,1]$ be any minimum resolving function of $T$. Let $M_2(T)=\{v_1, v_2, \ldots, v_a\}$, where $a\ge1$. For each $i\in\{1,2\ldots, a\}$, let $ter(v_i)=\sigma_i\ge2$, let $\ell_{i,1}, \ell_{i,2}, \ldots, \ell_{i,\sigma_i}$ be the terminal vertices of $v_i$, and let $N(v_i)=\{s_{i,1}, s_{i,2}, \ldots, s_{i, \sigma_i}\}$ such that $s_{i,j}$ lies on the $v_i-\ell_{i,j}$ path and we denote by $P^{i,j}$ the $s_{i,j}-\ell_{i,j}$ path, where $j\in\{1,2,\ldots, \sigma_i\}$.

For each $i\in\{1,2,\ldots, a\}$ and for any distinct $x,y\in\{1,2,\ldots, \sigma_i\}$, note that $R\{s_{i,x}, s_{i,y}\}=V(P^{i,x})\cup V(P^{i,y})$, and thus $g(V(P^{i,x}))+g(V(P^{i,y}))\ge 1$; by summing over the $\sigma_i \choose 2$ inequalities, we have $(\sigma_i-1)\sum_{j=1}^{\sigma_i}g(V(P^{i,j}))\ge {\sigma_i \choose 2}$, which implies $g(V(T_{v_i}))-g(v_i)=\sum_{j=1}^{\sigma_i}g(V(P^{i,j}))\ge\frac{\sigma_i}{2}$. So, $\sum_{i=1}^{a}(g(V(T_{v_i}))-g(v_i))\ge \sum_{i=1}^{a}\frac{\sigma_i}{2}=\frac{1}{2} (\sigma(T)-ex_1(T))=\dim_f(T)$. Since $g$ is a minimum resolving function of $T$, $g(V(T))=\dim_f(T)$; thus, $g(V(T_{v_i}))-g(v_i)=\frac{\sigma_i}{2}$ for each $i\in\{1,2,\ldots, a\}$ and $\sum_{i=1}^{a}(g(V(T_{v_i}))-g(v_i))=\dim_f(T)$.~\hfill
\end{proof}

Lemma~\ref{obs_tree1} implies the following

\begin{corollary}\label{obs_tree2}
Let $T$ be a tree with $ex(T)\ge1$, and let $g:V(T)\rightarrow [0,1]$ be any minimum resolving function of $T$. If $x$ is a major vertex, an interior degree-two vertex, or a vertex belonging to $T_v$ for $v\in M_1(T)$, then $g(x)=0$.
\end{corollary}

\begin{lemma}\label{obs_tree3}
Let $T$ be a tree with $ex(T)\ge1$ satisfying $\dim_{1,f}(T)=\dim_f(T)$. 
\begin{itemize}
\item[(a)] If $v\in M_2(T)$, then every terminal vertex of $v$ is adjacent to $v$ in $T$. 
\item[(b)] $T$ contains no major vertex of terminal degree one.
\item[(c)] $T$ contains neither a major vertex of terminal degree zero nor an interior degree-two vertex.
\end{itemize}
\end{lemma}

\begin{proof}
Let $T$ be a tree with $ex(T)\ge 1$ satisfying $\dim_{1,f}(T)=\dim_f(T)$. Let $g:V(T)\rightarrow [0,1]$ be any minimum resolving function of $T$, and let $h:V(T)\rightarrow [0,1]$ be any minimum $1$-truncated resolving function of $T$.

\medskip

(a) Let $v\in M_2(T)$ with $ter(v)=\alpha\ge2$. Let $\ell_1, \ell_2, \ldots, \ell_{\alpha}$ be the terminal vertices of $v$ in $T$. Suppose $d(v, \ell_i)\ge 2$ for some $i\in\{1,2,\ldots, \alpha\}$. By relabeling the vertices of $T$ if necessary, let $d(v, \ell_1)\ge2$.

First, let $\alpha\ge 3$. Let the $v-\ell_1$ path be given by $v=s_0, s_1, \ldots, s_d=\ell_1$, where $d\ge2$, and we denote by $P$ the $v-\ell_1$ path in $T$. Since $R_1\{s_{d-1}, s_d\}=\{s_{d-2}, s_{d-1}, s_d\} \subseteq V(P)$, $h(V(P))\ge 1$. Since $\dim_{1,f}(T)=\dim_f(T)$ and $\alpha\ge3$, $h(V(T_v))-h(V(P))\ge \frac{\alpha-1}{2}$ using the argument in the proof for Lemma~\ref{obs_tree1}. So, $h(V(T_v))\ge \frac{\alpha+1}{2}>\frac{\alpha}{2}=g(V(T_v))$. By Lemma~\ref{obs_tree1} and Corollary~\ref{obs_tree2}, $\dim_{1,f}(T)>\dim_f(T)$. So, $\dim_{1,f}(T)=\dim_f(T)$ implies that $d(v, \ell_i)=1$ for each $i\in\{1,2,\ldots, \alpha\}$.

Second, let $\alpha=2$. Let the $v-\ell_1$ path and the $v-\ell_2$ path, respectively, be given by $v=s_0, s_1, \ldots, s_d=\ell_1$ and $v=s_0, s'_1, \ldots, s'_t=\ell_2$, where $d\ge2$; let $P^1$ and $P^2$ denote the $s_1-\ell_1$ path and the $s'_1-\ell_2$ path, respectively. Note that $R_1\{s_{d-1}, s_d\} \subseteq \{v\} \cup V(P^1)$, $R_1\{s_1,s'_1\}\subseteq V(P^1) \cup V(P^2)$, $R_1\{v, s_2\}\subseteq (N[v] \cup V(P^1))-\{s_1\}$ and $R_1\{v, s'_1\} \subseteq N[v] \cup V(P^2)$. So, $h(v)+h(V(P^1))\ge1$, $h(V(P^1))+h(V(P^2))\ge1$, $h(N[v] \cup V(P^1))-h(s_1)\ge1$ and $h(N[v] \cup V(P^2))\ge1$. Since each vertex in $N[v] \cup V(T_v)$ appears at most $3$ times in the $4$ inequalities, by summing over the $4$ inequalities, we have $3h(N[v] \cup V(T_v))\ge 4$. Since each vertex in $N[v]-\{s_1, s'_1\}$ is either a major vertex or an interior degree-two vertex in $T$, $g(N[v]-\{s_1, s'_1\})=0$ by Corollary~\ref{obs_tree2}. So, $h(N[v] \cup V(T_v))\ge \frac{4}{3}>1=g(N[v] \cup V(T_v))=g(V(T_v)-\{v\})$. Thus, $\dim_{1,f}(T)=\dim_f(T)$ implies that $d(v, \ell_1)=1=d(v,\ell_2)$.

\medskip

(b) Let $v\in M_1(T)$, and let $\ell$ be the terminal vertex of $v$ in $T$ such that the $v-\ell$ path is given by $v=s_0, s_1, \ldots, s_d=\ell$. If $d(v, \ell)\ge2$, then $R_1\{s_{d-1}, s_d\}\subseteq V(T_v)$ and $h(V(T_v))\ge1>0=g(V(T_v))$ by Corollary~\ref{obs_tree2}. If $d(v, \ell)=1$, then $R_1\{v, \ell\}=N[v]$ and $h(N[v])\ge1>0=g(N[v])$ since each vertex in $N[v]$ is a major vertex, an interior degree-two vertex, or a vertex belonging to $T_v$ for $v\in M_1(T)$. So, $\dim_{1,f}(T)=\dim_f(T)$ implies that $M_1(T)=\emptyset$.

\medskip

(c) Suppose $x$ is either an interior degree-two vertex or a major vertex of terminal degree zero in $T$. Then there exist distinct $w,w'\in M_2(T)$ such that $x$ lies on the $w-w'$ path and that $w$ and $w'$ are the only exterior major vertices on the $w-w'$ path. By relabeling the vertices of $T$ if necessary, we may assume that $xw\in E(T)$. If $\ell_1, \ell_2, \ldots, \ell_{\alpha}$ are the terminal vertices of $w$, where $\alpha\ge2$, then $\ell_iw\in E(T)$ for each $i\in\{1,2,\ldots, \alpha\}$ by (a) of the current lemma. 

For each $i\in\{1,2,\ldots, \alpha\}$, $R_1\{\ell_i, x\}=\{\ell_i\} \cup (N[x]-\{w\})$ and $h(\ell_i)+h(N[x]-\{w\})\ge 1$. By summing over the $\alpha$ inequalities, we have $(\sum_{i=1}^{\alpha}h(\ell_i))+\alpha h(N[x]-\{w\}) \ge \alpha$. Since each vertex in $N[x]-\{w\}$ is either a major vertex or an interior degree-two vertex in $T$, $g(N[x]-\{w\})=0$ by Corollary~\ref{obs_tree2}. So, we have $(\sum_{i=1}^{\alpha}h(\ell_i))+\alpha h(N[x]-\{w\})\ge \alpha>\frac{\alpha}{2}=(\sum_{i=1}^{\alpha}g(\ell_i))+\alpha g(N[x]-\{w\})=g(V(T_w)-\{w\})$. So, $\dim_{1,f}(T)=\dim_f(T)$ implies that $T$ contains neither a major vertex of terminal degree zero nor an interior degree-two vertex.~\hfill
\end{proof}

\begin{proposition}
Let $T$ be a non-trivial tree. Then $\dim_{1,f}(T)=\dim_f(T)$ if and only if $T\in\{P_2, P_3\}$, or $ex(T)\ge 1$ and $V(T)=M_2(T) \cup L(T)$. 
\end{proposition}

\begin{proof}
Let $T$ be a tree of order $n\ge2$.

($\Leftarrow$) If $T\in\{P_2, P_3\}$, then $\dim_{1,f}(T)=1=\dim_f(T)$ by Theorems~\ref{characterization_mixed}(c) and~\ref{frac_kdim=1}. Now, suppose $ex(T)\ge 1$ and $V(T)=M_2(T) \cup L(T)$. If $g:V(T)\rightarrow [0,1]$ is a function defined by $g(u)=\frac{1}{2}$ for each $u\in L(T)$ and $g(w)=0$ for each $w\in M_2(T)$, then $g$ is a $1$-truncated resolving function of $T$ with $g(V(T))=\frac{1}{2}\sigma(T)$; thus $\dim_{1,f}(T) \le \frac{1}{2}\sigma(T)$. Since $\dim_{1,f}(T)\ge \dim_f(T)=\frac{1}{2} \sigma(T)$ by Observation~\ref{obs_bounds}(d) and Theorem~\ref{thm_frac}(a), $\dim_{1,f}(T)=\frac{1}{2}\sigma(T)=\dim_f(T)$.

($\Rightarrow$) Let $\dim_{1,f}(T)=\dim_f(T)$. If $ex(T)=0$, then $T=P_n$ and $\dim_f(P_n)=1$ by Theorem~\ref{characterization_mixed}(c). So, $\dim_{1,f}(P_n)=\dim_f(P_n)$ implies $\dim_{1,f}(P_n)=1$, and thus $n\in\{2, 3\}$ by Theorem~\ref{frac_kdim=1}. 

Now, suppose $ex(T)\ge1$. By Lemma~\ref{obs_tree3}(b)(c), $T$ contains neither an interior degree-two vertex nor a major vertex with terminal degree at most one. Moreover, by Lemma~\ref{obs_tree3}(a), $T$ contains no exterior degree-two vertex. So, each vertex in $T$ is either a major vertex with terminal degree at least two, or a leaf that is adjacent to some exterior major vertex in $T$.~\hfill
\end{proof}

For any $k\in\mathbb{Z}^+$, it's an interesting yet a challenging task to characterize all connected graphs $G$ satisfying $\dim_k(G)=\dim(G)$ or $\dim_{k,f}(G)=\dim_f(G)$ even when $G$ is restricted to trees. We recall the following result on trees $T$ satisfying $\dim_1(T)=\dim(T)$.

\begin{proposition}\emph{\cite{broadcast}}
For any non-trivial tree $T$, $\dim_1(T)=\dim(T)$ if and only if $T\in\{P_2, P_3\}$ or $T$ is a tree obtained from the star $K_{1,x}$ ($x \ge 3$) by subdividing at most $x-1$ edges exactly once.
\end{proposition}

Now, for trees $T$ with $ex(T)=1$, we characterize $T$ satisfying $\dim_{k,f}(T)=\dim_f(T)$.

\begin{proposition}
Let $k\in\mathbb{Z}^+$, and let $T$ be a tree with $ex(T)=1$ such that $\ell_1, \ell_2, \ldots, \ell_{\alpha}$ are the terminal vertices of the exterior major vertex $v$ in $T$. Then $\dim_{k,f}(T)=\dim_f(T)$ if and only if $d(v, \ell_i) \le k$ for each $i\in\{1,2,\ldots,\alpha\}$.
\end{proposition}

\begin{proof}
Let $k\in\mathbb{Z}^+$, and let $T$ be a tree with $ex(T)=1$. Let $v$ be the exterior major vertex of $T$, $N(v)=\{s_1, s_2, \ldots, s_{\alpha}\}$, and let $\ell_1, \ell_2, \ldots, \ell_{\alpha}$ be the terminal vertices of $v$ in $T$ such that $s_i$ lies on the $v-\ell_i$ path for each $i\in\{1,2,\ldots, \alpha\}$, where $\alpha\ge3$. Let $P^i$ denote the $s_i-\ell_i$ path, where $i\in\{1,2,\ldots, \alpha\}$.

($\Leftarrow$) Suppose $d(v, \ell_i) \le k$ for each $i\in\{1,2,\ldots,\alpha\}$. Let $g: V(T)\rightarrow[0,1]$ be a function defined by 
\begin{equation*}
g(u)=\left\{
\begin{array}{ll}
\frac{1}{2} & \mbox{ if } u\in \{s_1, s_2, \ldots, s_{\alpha}\},\\
0 & \mbox{ otherwise}.
\end{array}\right.
\end{equation*}
We show that $g$ is a $k$-truncated resolving function for $T$. Let $x$ and $y$ be distinct vertices in $T$. First, suppose $x$ and $y$ lie on the $v-\ell_i$ path for some $i\in\{1,2,\ldots, \alpha\}$; let $d(v,x)<d(v,y)$ without loss of generality. Since $d(x,s_j) \le k$ and $d_k(x,s_j)\neq d_k(y,s_j)$ for each $j\in\{1,2,\ldots, \alpha\}-\{i\}$, $R_k\{x,y\}\supseteq N(v)-\{s_i\}$ and $g(R_k\{x,y\}) \ge g(N(v))-g(s_i)=\frac{\alpha-1}{2}\ge1$ since $\alpha\ge3$. Second, suppose $x$ lies on the $s_i-\ell_i$ path and $y$ lies on the $s_j-\ell_j$ path for distinct $i,j\in\{1,2,\ldots, \alpha\}$. If $d(v, x)=d(v,y)$, then $R_k\{x,y\}=V(P^i)\cup V(P^j)$ and $g(R_k\{x,y\})\ge g(V(P^i))+g(V(P^j))\ge g(s_i)+g(s_j)=1$. If $d(v, x)\neq d(v,y)$, say $d(v,x)<d(v,y)$, then $d(x, s_t) \le k$ and $d_k(x,s_t)\neq d_k(y,s_t)$ for each $t\in\{1,2,\ldots, \alpha\}-\{j\}$; thus, $R_k\{x,y\} \supseteq N(v)-\{s_j\}$ and $g(R_k\{x,y\})\ge g(N(v))-g(s_j)\ge \frac{\alpha-1}{2}\ge1$. So, $g$ is a $k$-truncated resolving function of $T$ with $g(V(T))=\frac{\alpha}{2}$, and thus $\dim_{k,f}(T) \le \frac{\alpha}{2}=\dim_f(T)$ by Theorem~\ref{thm_frac}(a). Since $\dim_{k,f}(T)\ge\dim_f(T)$ by Observation~\ref{obs_bounds}(c), we have $\dim_{k,f}(T)=\frac{\alpha}{2}=\dim_f(T)$.

($\Rightarrow$) Let $\dim_{k,f}(T)=\dim_f(T)$. Assume, to the contrary, that $d(v, \ell_i)\ge k+1$ for some $i\in\{1,2,\ldots, \alpha\}$. By relabeling the vertices of $T$ if necessary, let $d(v, \ell_1)\ge k+1$ and let $v-\ell_1$ path be given by $v=u_0, s_1=u_1, u_2, \ldots, u_k, u_{k+1},\ldots, u_z=\ell_1$. Let $h: V(T)\rightarrow [0,1]$ be a minimum $k$-truncated resolving function of $T$. Since $R_k\{u_{z-1}, u_z\} \subseteq \{v\} \cup V(P^1)$, $h(v)+h(V(P^1))\ge h(R_k\{u_{z-1}, u_z\})\ge 1$. Note that, for any distinct $i, j\in\{2,3,\ldots, \alpha\}$, $R_k\{s_i, s_j\}\subseteq V(P^i) \cup V(P^j)$ and $h(V(P^i))+h(V(P^j))\ge1$; thus, $\sum_{i=2}^{\alpha}h(V(P^i)) \ge \frac{\alpha-1}{2}$. So, $\dim_{k,f}(T)=h(V(T))=h(v)+\sum_{i=1}^{\alpha}h(V(P^i))\ge 1+\frac{\alpha-1}{2}=\frac{\alpha+1}{2}>\frac{\alpha}{2}=\dim_f(T)$, which contradicts the assumption that $\dim_{k,f}(T)=\dim_f(T)$. So, $d(v, \ell_i)\le k$ for each $i\in\{1,2,\ldots, \alpha\}$.~\hfill
\end{proof}


\section{Comparison on $\dim(G)$, $\dim_f(G)$, $\dim_k(G)$ and $\dim_{k,f}(G)$}\label{Sec_comparison}

In this section, we examine the relation among $\dim_f(G)$, $\dim_{k,f}(G)$, $\dim(G)$ and $\dim_k(G)$ for $k\in\mathbb{Z}^+$ in conjunction with Observation~\ref{obs_bounds}. We show that, for two connected graphs $H$ and $G$ with $H \subset G$, $\frac{\dim_{k,f}(H)}{\dim_{k,f}(G)}$ can be arbitrarily large. We also show the existence of non-isomorphic graphs $G$ and $H$ with $\dim_k(G)=\dim_k(H)$ and $\dim_{k,f}(G) \neq \dim_{k,f}(H)$. We conclude the paper with some open problems.

It is known that metric dimension is not a monotone parameter on subgraph inclusion (see~\cite{linegraph}), and the following results were obtained in~\cite{broadcast, distKdim}. 

\begin{theorem}
Let $H$ and $G$ be connected graphs with $H\subset G$. Then
\begin{itemize}
\item[(a)] \emph{\cite{broadcast}} $\frac{\dim(H)}{\dim(G)}$ and $\frac{\dim_1(H)}{\dim_1(G)}$ can be arbitrarily large;
\item[(b)] \emph{\cite{distKdim}} for any $k\in \mathbb{Z}^+$, $\frac{\dim_k(H)}{\dim_k(G)}$ can be arbitrarily large.
\end{itemize}
\end{theorem}

We recall the following construction from~\cite{broadcast}. For $m \ge 3$, let $H=K_{\frac{m(m+1)}{2}}$; let $V(H)$ be partitioned into $V_1, V_2, \ldots, V_m$ such that $V_i=\{w_{i,1}, w_{i,2}, \ldots, w_{i,i}\}$ with $|V_i|=i$, where $i \in \{1,2,\ldots,m\}$. Let $G$ be the graph obtained from $H$ and $m$ isolated vertices $u_1,u_2, \ldots, u_m$ such that, for each $i \in \{1,2,\ldots, m\}$, $u_i$ is joined by an edge to each vertex of $V_i \cup (\cup_{j=i+1}^{m}\{w_{j,i}\})$; notice $H \subset G$. Since $\diam(H)=1$ and $\diam(G)=2$, by Observation~\ref{obs_diam}(b), $\dim_{k,f}(H)=\dim_f(H)$ and $\dim_{k,f}(G)=\dim_f(G)$ for any $k\in\mathbb{Z}^+$. Note that $\dim_f(H) =\frac{m(m+1)}{4}$ by Theorem~\ref{thm_frac}(e), and $\dim_f(G) \le m$ by Observation~\ref{obs_bounds}(a) since $\{u_1, u_2, \ldots, u_m\}$ forms a resolving set of $G$. So, $\frac{\dim_{k,f}(H)}{\dim_{k,f}(G)}=\frac{\dim_f(H)}{\dim_f(G)} \ge \frac{m+1}{4}$ for any $k\in\mathbb{Z}^+$, which implies the following 

\begin{corollary}
For any $k\in\mathbb{Z}^+$, there exist connected graphs $H$ and $G$ such that $H \subset G$ and both $\frac{\dim_f(H)}{\dim_f(G)}$ and $\frac{\dim_{k,f}(H)}{\dim_{k,f}(G)}$ can be arbitrarily large.
\end{corollary}

It was shown that $\frac{\dim_k(G)}{\dim(G)}$ and $\frac{\dim_{k,f}(G)}{\dim_f(G)}$ can be arbitrarily large (see~\cite{distKdim} and Theorem~\ref{comp_fkdim}, respectively). In view of Observation~\ref{obs_bounds}, it is easy to see that $\dim(G)-\dim_f(G)$ and $\dim_k(G)-\dim_{k,f}(G)$ can be arbitrarily large. We also show that $\dim(G)-\dim_{k,f}(G)$ and $\dim_{k,f}(G)-\dim(G)$ can be arbitrarily large.

\begin{theorem}\emph{\cite{tree1, tree2, tree3}}\label{dim_tree}
For any tree $T$ that is not a path, $\dim(T)=\sigma(T)-ex(T)$.
\end{theorem}

\begin{remark}\label{remark_cycle}
Let $k\in\mathbb{Z}^+$, and let $G$ be a non-trivial connected graph. Then
\begin{itemize}
\item[(a)] $\dim(G)-\dim_f(G)$ can be arbitrarily large;
\item[(b)] $\dim_k(G)-\dim_{k,f}(G)$ can be arbitrarily large;
\item[(c)] $\dim(G)-\dim_{k,f}(G)$ can be arbitrarily large;
\item[(d)] $\frac{\dim_{k,f}(G)}{\dim(G)}$ can be arbitrarily large.
\end{itemize}
\end{remark}

\begin{proof}
Let $k\in\mathbb{Z}^+$. For (a), (b) and (c), let $G$ be a tree with $V(G)=M_2(G) \cup L(G)$ such that $M_2(G)=\{v_1, v_2, \ldots, v_x\}$ is the set of exterior major vertices of $G$ with $ter(v_i)=\alpha\ge3$ for each $i\in\{1,2,\ldots, x\}$, where $x\ge1$; see Figure~\ref{fig_tree}. Then $\dim_k(G) \ge \dim(G)=x(\alpha-1)$ by Observation~\ref{obs_bounds}(b) and Theorem~\ref{dim_tree}. Also, note that $\dim_{k,f}(G) \ge \dim_{f}(G)=\frac{x\alpha}{2}$ by Observation~\ref{obs_bounds}(c) and Theorem~\ref{thm_frac}(a). Since a function $g:V(G)\rightarrow[0,1]$ defined by $g(u)=\frac{1}{2}$ for each $u\in L(G)$ and $g(w)=0$ for each $w\in M_2(G)=V(G)-L(G)$ is a $k$-truncated resolving function for $G$ with $g(V(G))=\frac{x\alpha}{2}$, $\dim_{k,f}(G)\le \frac{x\alpha}{2}$. So, $\dim_{k,f}(G) = \dim_{f}(G)=\frac{x\alpha}{2}$. Thus, $\dim_k(G)-\dim_{k,f}(G)\ge \dim(G)-\dim_{k,f}(G)=\dim(G)-\dim_f(G)=x(\alpha-1)-\frac{x\alpha}{2}=\frac{x(\alpha-2)}{2} \rightarrow \infty$ as $x\rightarrow\infty$ or $\alpha\rightarrow\infty$. 

For (d), let $G=C_{x(k+1)(k+5)}$ for some $x\in \mathbb{Z}^+$. Then $\dim_{k,f}(G)=\frac{x(k+5)}{2}$ by Theorem~\ref{fkdim_cycle} and $\dim(G)=2$. So, $\frac{\dim_{k,f}(G)}{\dim(G)}=\frac{x(k+5)}{4}\rightarrow\infty$ as $x\rightarrow\infty$.~\hfill
\end{proof}

\begin{figure}[ht]
\centering
\begin{tikzpicture}[scale=.65, transform shape]

\node [draw, shape=circle, scale=.8] (1) at  (0,0) {};
\node [draw, shape=circle, scale=.8] (2) at  (3,0) {};
\node [draw, shape=circle, scale=.8] (3) at  (6,0) {};
\node [draw, shape=circle, scale=.8] (4) at  (9.5,0) {};

\node [scale=1.2] at (0,0.35) {$v_1$};
\node [scale=1.2] at (3,0.35) {$v_2$};
\node [scale=1.2] at (6,0.35) {$v_3$};
\node [scale=1.2] at (9.5,0.35) {$v_x$};

\node [draw, shape=circle, scale=.8] (a1) at  (-0.8,-1) {};
\node [draw, shape=circle, scale=.8] (a2) at  (0,-1) {};
\node [draw, shape=circle, scale=.8] (a3) at  (1,-1) {};
\node [draw, shape=circle, scale=.8] (b1) at  (2.3,-1) {};
\node [draw, shape=circle, scale=.8] (b2) at  (3,-1) {};
\node [draw, shape=circle, scale=.8] (b3) at  (4,-1) {};
\node [draw, shape=circle, scale=.8] (c1) at  (5.2,-1) {};
\node [draw, shape=circle, scale=.8] (c2) at  (6,-1) {};
\node [draw, shape=circle, scale=.8] (c3) at  (7,-1) {};
\node [draw, shape=circle, scale=.8] (d1) at  (8.7,-1) {};
\node [draw, shape=circle, scale=.8] (d2) at  (9.5,-1) {};
\node [draw, shape=circle, scale=.8] (d3) at  (10.5,-1) {};

\node [scale=1.2] at (-0.7,-1.4) {$\ell_{1,1}$};
\node [scale=1.2] at (0.1,-1.4) {$\ell_{1,2}$};
\node [scale=1.2] at (1.1,-1.4) {$\ell_{1,\alpha}$};
\node [scale=1.2] at (2.4,-1.4) {$\ell_{2,1}$};
\node [scale=1.2] at (3.1,-1.4) {$\ell_{2,2}$};
\node [scale=1.2] at (4.1,-1.4) {$\ell_{2,\alpha}$};
\node [scale=1.2] at (5.3,-1.4) {$\ell_{3,1}$};
\node [scale=1.2] at (6.1,-1.4) {$\ell_{3,2}$};
\node [scale=1.2] at (7.1,-1.4) {$\ell_{3,\alpha}$};
\node [scale=1.2] at (8.8,-1.4) {$\ell_{x,1}$};
\node [scale=1.2] at (9.6,-1.4) {$\ell_{x,2}$};
\node [scale=1.2] at (10.6,-1.4) {$\ell_{x,\alpha}$};

\draw(1)--(2)--(3)--(7,0);\draw[thick, dotted](7,0)--(8.5,0);\draw(8.5, 0)--(4);
\draw(a1)--(1)--(a2);\draw(1)--(a3);\draw[thick, dotted](0.3,-1)--(0.7,-1);
\draw(b1)--(2)--(b2);\draw(2)--(b3);\draw[thick, dotted](3.3,-1)--(3.7,-1);
\draw(c1)--(3)--(c2);\draw(3)--(c3);\draw[thick, dotted](6.3,-1)--(6.7,-1);
\draw(d1)--(4)--(d2);\draw(4)--(d3);\draw[thick, dotted](9.8,-1)--(10.2,-1);

\end{tikzpicture}
\caption{Trees $G$ such that $\dim(G)-\dim_f(G)$, $\dim_k(G)-\dim_{k,f}(G)$ and $\dim(G)-\dim_{k,f}(G)$ can be arbitrarily large.}
\label{fig_tree}
\end{figure}
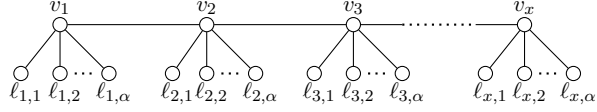

Next, we show the existence of non-isomorphic graphs $G$ and $H$ with $\dim_k(G)=\dim_k(H)$ and $\dim_{k,f}(G)\neq\dim_{k,f}(H)$ for each $k\in\mathbb{Z}^+$. We recall the following result.

\begin{theorem}\emph{\cite{distKdim}}\label{kdim_pathcycle}
Let $k\in\mathbb{Z}^+$. Then 
\begin{itemize}
\item[(a)] $\dim_k(P_n)=1$ for $2\le n \le k+2$;
\item[(b)] $\dim_k(C_n)=2$ for $3\le n \le 3k+3$, and $\dim_k(P_n)=2$ for $k+3 \le n \le 3k+3$;
\item[(c)] for $n \ge 3k+4$, 
$$\dim_{k}(C_n)=\dim_k(P_n)=\left\{
\begin{array}{ll}
\lfloor\frac{2n+3k-1}{3k+2}\rfloor & \mbox{ if } n \equiv 0,1,\ldots, k+2 \pmod{(3k+2)},\\ 
\lfloor\frac{2n+4k-1}{3k+2}\rfloor & \mbox{ if } n \equiv k+3,\ldots,  \lceil \frac{3k+5}{2}\rceil-1 \pmod{(3k+2)},\\ 
\lfloor\frac{2n+3k-1}{3k+2}\rfloor & \mbox{ if } n \equiv \lceil \frac{3k+5}{2}\rceil,\ldots, 3k+1 \pmod{(3k+2)}.
\end{array}\right.$$
\end{itemize}
\end{theorem}

\begin{remark}
Let $k\in\mathbb{Z}^+$. There exist non-isomorphic graphs $G$ and $H$ such that $\dim_{k}(G)=\dim_{k}(H)$ and $\dim_{k,f}(G) \neq \dim_{k,f}(H)$. For $n \ge 3k+4$, $\dim_k(C_n)=\dim_k(P_n)$ by Theorem~\ref{kdim_pathcycle} and $\dim_{k,f}(C_n)\neq\dim_{k,f}(P_n)$ for $n\equiv 1 \pmod{(2k+2)}$ by Theorems~\ref{fkdim_cycle} and~\ref{frac_distK_path}(c).
\end{remark}

We conclude the paper with some open problems.

\begin{question}
For any tree $T$ and for any $k\in{Z}^+$, can we determine $\dim_{k,f}(T)$?
\end{question}

\begin{question}
For any $k\in\mathbb{Z}^+$, can we characterize all connected graphs $G$ satisfying $\dim_{k,f}(G)=\dim_f(G)$?
\end{question}

\begin{question}
For any $k\in\mathbb{Z}^+$, can we characterize all connected graphs $G$ satisfying $\dim_{k}(G)=\dim_{k,f}(G)$?
\end{question}

\end{document}